\documentclass[11pt,letterpaper]{amsart}

 \usepackage[a4paper]{geometry} 
\newcounter{commentcounter}

\usepackage{amsmath} 
\usepackage{amssymb} 
\usepackage{amsthm} 
\usepackage{stmaryrd} 
\usepackage[english]{babel} 
\usepackage[font=small,justification=centering]{caption} 
\usepackage[nodayofweek]{datetime}
\usepackage[shortlabels]{enumitem} 
\usepackage[T1]{fontenc} 
\usepackage[utf8]{inputenc} 
\usepackage{ifthen} 
\usepackage{mathabx} 
\usepackage{mathtools} 
\usepackage[dvipsnames]{xcolor} 
\usepackage[pdftex,  colorlinks=true]{hyperref} 
    \hypersetup{urlcolor=RoyalBlue, linkcolor=RoyalBlue,  citecolor=black}
\usepackage{setspace} 
\usepackage{tikz-cd} 
\usepackage{xfrac} 
\usepackage[capitalize]{cleveref} 

\usepackage{wrapfig}
\usepackage{caption}
\usepackage{subcaption}

\makeatletter
\@namedef{subjclassname@1991}{Mathematical subject classification 1991}
\@namedef{subjclassname@2000}{Mathematical subject classification 2000}
\@namedef{subjclassname@2010}{Mathematical subject classification 2010}
\@namedef{subjclassname@2020}{Mathematical subject classification 2020}
\makeatother

\newtheorem{theorem}{Theorem}[section]
\newtheorem{lemma}[theorem]{Lemma}
\newtheorem{corollary}[theorem]{Corollary}
\newtheorem{proposition}[theorem]{Proposition}

\newtheorem{thmx}{Theorem}
\newtheorem{corx}[thmx]{Corollary}
\newtheorem{propx}[thmx]{Proposition}

\theoremstyle{definition}
\newtheorem{definition}[theorem]{Definition}

\theoremstyle{plain}

    \newtheoremstyle{TheoremNum}
        {\topsep}{\topsep} 
        {\itshape} 
        {-0.25cm} 
        {\bfseries} 
        {.} 
        { }  
        {\thmname{#1}\thmnote{ \bfseries #3}}
    \theoremstyle{TheoremNum}
    \newtheorem{duplicate}{}

\newcommand*{\claimproofname}{My proof}

\DeclareMathOperator{\Out}{\mathrm{Out}}

\usepackage{tikz}
\usetikzlibrary{arrows,quotes}
\tikzstyle{blackNode}=[fill=black, draw=black, shape=circle]

\title{Growth rates, stable subgroups, and regular languages}

\usepackage[foot]{amsaddr}

\author{Kaitlin Ragosta}
\email{kaitlin.ragosta@ehu.eus}

\date{\today}
\subjclass[2020]{}

\begin{document}

\begin{abstract}
    We show that the language of geodesic words representing elements of a stable subgroup $H$ of a group $G$ with finite generating set $A$ is regular, and that there is a sublanguage which bijects $H$. Consequently, the growth function of $H$ with respect to $A$ is rational, and in many cases, one can deduce a growth rate gap between $H$ and $G$. In particular, this applies to convex cocompact subgroups of $\mathrm{Out}(F_n)$, handlebody groups, and Torelli groups of surfaces of sufficient complexity. We also provide an example of a finitely presented, relatively hyperbolic, and Morse local-to-global group which contains a stable subgroup with unsolvable membership problem, answering a question of Cordes, Russell, Spriano, and Zalloum. 
\end{abstract}

\maketitle

\section{Introduction}

Stable subgroups were introduced in \cite{Durham-Taylor-stability} to generalize quasi-convex subgroups of a hyperbolic group. They have garnered interest both because they encompass geometrically important classes of subgroups, including convex cocompact subgroups of $\mathrm{MCG}(S)$ and $\Out(F_n)$, and because several properties of quasi-convex subgroups have been shown to hold for stable subgroups. There has been particular success in the case where the ambient group is Morse local-to-global \cite{CRSZ,Growth-rate-gap-all-Mltg}. In this note, we show that the Morse local-to-global assumption can be removed from some of those results.
\begin{thmx}\label{thm:regular-language}
        Let $G$ be a group with finite generating set $A$, and let $H \leq G$ be a subgroup. Let $L_H$ be the language of geodesic words representing elements of $H$. The subgroup $H$ is stable if and only if $L_H$ is a regular language all of whose elements are $M$-Morse for some Morse gauge $M$. Moreover, if $H$ is stable, there is a sublanguage $J_H \subseteq L_H$ such that $J_H$ is regular and each element of $H$ is represented by exactly one word in $J_H$. 
\end{thmx}
\begin{corx}\label{corx:rational-growth}
    Let $H$ be a stable subgroup of a group $G$ with finite generating set $A$, and let $f_{H,A}(n) = |B_n(e) \cap H|$. There exist polynomials $P(x)$ and $Q(x) \in \mathbb{Q}[x]$ such that 
    \begin{align*}
        \sum_{n=0}^{\infty} f_{H,A} \cdot x^n = \frac{P(x)}{Q(x)}\text{.}
    \end{align*}
\end{corx}
Theorem \ref{thm:regular-language} gives a regular language characterization of the geometric property of stability, providing the first step towards resolving \cite[Question~1]{CRSZ}, which asked whether a subgroup is stable if and only if the language of $(k,c)$ quasi-geodesic words that start at the identity and end in $H$ is regular whenever $k$ is rational. This is the case for hyperbolic groups \cite{Holt-Rees-qgeod-regular,Sam-Patrick-Davide-qgeod-regular}. 

\cite[Question~5]{CRSZ} asked whether the growth rate of an infinite-index stable subgroup is always smaller than that of the ambient group; this was previously known when the ambient group is Morse local-to-global or admits a path system with a constricting element\cite{DFW,CRSZ,Growth-rate-gap-all-Mltg,Abdul-Alex-injective-contracting,Legaspi-path-system-growth}. We show that the non-positive curvature assumptions on $G$ can be removed if $H$ satisfies a sufficiently strong combination theorem.

\begin{thmx}\label{thmx:growth-rate-gap}
    Let $G$ be a group with finite generating set $A$, and let $H$ be a stable subgroup. Suppose there is an infinite-order element $g \in G$ such that $\langle H , g \rangle \cong H * \langle g \rangle$. The growth rate of $H$ with respect to $A$ is strictly smaller than the growth rate of $G$ with respect to $A$.
\end{thmx}

From this, we deduce a number of interesting examples.
\begin{corx}\label{corx:transverse-examples}
    For any finite set $A$ of $G$, $\lambda_{H,A} < \lambda_{G, A}$ in the following cases:
    \begin{enumerate}
        \item $G = \Out(\mathbb{F}_n)$ for $n \geq 3$ and $H$ is a convex cocompact subgroup.
        \item $G$ is a Torelli group $\mathcal{T}_g$ for $g > 2$ and $H$ is a convex cocompact subgroup of $\mathrm{MCG}(S_g)$ with $H < \mathcal{T}_g$.
        \item G is a genus $g$ handlebody group $\mathcal{H}_g$ for $g > 1$ and $H$ is a convex cocompact subgroup of $\mathrm{MCG}(\partial V_g)$ with $H < \mathcal{H}_g$.
    \end{enumerate}
\end{corx}

Finally, \cite[Question~3]{CRSZ} asked whether the membership problem is decidable for stable subgroups of Morse local-to-global groups and if their Theorem E, analogous to our Theorem \ref{thm:regular-language} in the Morse local-to-global case, had implications for the complexity of such an algorithm. In the final section, we show that this is not the case, even under additional desirable conditions.

\begin{propx}\label{propx:counterexample}
    There is a finitely presented, relatively hyperbolic, and Morse local-to-global group $G$ such that $G$ has a stable subgroup with undecidable membership problem. 
\end{propx}

\subsection{Acknowledgments}
 I would like to thank Carolyn Abbott for her mentorship and Abdul Zalloum for asking whether the properties in Theorem \ref{thm:regular-language} and Corollary \ref{corx:rational-growth} are intrinsic properties of stable subgroups. I would also like to thank Jacob Russell, David Futer, and Joshua Perlmutter for helpful conversations. I am grateful to my group at the 2026 Junior Research Retreat on Separability and Morseness in Kopp, namely Francesco Fournier-Facio, Huaitao Gui, Rafaela Ioannou, Joshua Perlmutter, Talia Schlomovich, and Bratati Som, for discussions related to the final section, and to the organizers of the retreat. I would especially like to thank Abdul Zalloum and Francesco Fournier Facio for several key observations and suggestions related to Theorem \ref{thm:regular-language} and Proposition \ref{propx:counterexample} respectively.

\section{Background}
Throughout the paper, generating sets are assumed to be symmetric. We first recall relevant definitions, beginning with stable subgroups.

\begin{definition}\cite[Definition~2.22]{CRSZ}
    Let $G$ be a group with a finite generating set $A$, let $M$ be a Morse gauge, and let $k\geq 0$. A subgroup $H\leq G$ is \textit{$(M,k)$-stable} in $\text{Cay}(G,A)$, if for every $h\in H$, every geodesic in $\text{Cay}(G,A)$ from $e$ to $h$ is $M$-Morse and contained in the $k$-neighborhood of $H$. A subgroup $H\leq G$ is a \textit{stable subgroup} if for any choice of finite generating set $A$ for $G$, there exist $M$ and $k$ such that $H$ is $(M,k)$-stable in $\text{Cay}(G,A)$.
\end{definition}
While this is stated differently than the original definition of stability in \cite{Durham-Taylor-stability}, the two are equivalent \cite{Cordes-Hume-stability}. We also recall the definition of the growth rate of a group or subgroup with respect to a particular finite generating set. 

\begin{definition}
    Let $G$ be a group with finite generating set $A$, and let $H \leq G$. The \emph{growth function} of $H$ with respect to $A$ is the function $f_{H,A}: \mathbb{N} \rightarrow \mathbb{N}$ where $f_{H,A}(n) = B_{n,A}(e) \cap H$, where $B_{n,A}$ is the ball of radius $n$ around the identity in $\text{Cay}(G,A)$. The \emph{growth rate} of $H$ with respect to $A$ is
\begin{align*}
    \lambda_{H,A} := \limsup_{n \rightarrow \infty} \sqrt[n]{f_{H,A}(n)}\text{.}
\end{align*}
\end{definition}

In general, both the growth rate and the growth function are highly dependent on the choice of generating set $A$. We now recall several facts about languages and finite-state automata.
\begin{definition}
    Let $A$ be a finite set and $A^{\star}$ be the free monoid over $A$. A \emph{word} in $A$ is an element $w \in A^{\star}$. If $w = a_1 a_2 \cdots a_n$, the elements $a_i$ are the \emph{letters} of $w$. A \emph{language with alphabet $A$} is a set of words in $A^{\star}$. The \emph{word length} of $w$ is the number of letters which appear in it.
\end{definition}

When $A$ is the generating set of a group $G$, a word $w \in A^{\star}$ can also be viewed as an element of $G$. We call this element $\overline{w}$. Every path in $\text{Cay}(G,A)$ produces a word in $A^{\star}$ by concatenating labels of edges in the order they appear along the path, and conversely, every word in $A^{\star}$ produces a path in $\text{Cay}(G,A)$ by starting at the identity and following the path dictated by the letters in the word from left to right. A word in $A^{\star}$ is \emph{geodesic} if the associated path is a geodesic in $\text{Cay}(G,A)$ or, equivalently, if $w$ is a minimal length representative of $\overline{w}$ with respect to the generating set $A$. 

\begin{definition}
    Let $A$ be a finite set. A \emph{finite state automaton} with alphabet $A$ is a tuple $\mathcal{G} = (\Gamma, A, Y, s_0)$ where:
    \begin{itemize}
        \item $\Gamma$ is a finite directed graph. The vertices of $\Gamma$ are the \emph{states} of $\mathcal{G}$.
        \item Each edge of $\Gamma$ is labeled by an element of the alphabet $A$.
        \item $Y$ is a subset of the vertices of $\Gamma$. The vertices of $Y$ are the \emph{accept states} of $\mathcal{G}$, while the vertices not in $Y$ are the \emph{reject states}.
        \item $s_0$ is a vertex of $\Gamma$. We call $s_0$ the \emph{initial state} of $\mathcal{G}$.
    \end{itemize}
\end{definition}

\begin{definition}
    A \emph{directed path} in a finite state automaton $\mathcal{G} = (\Gamma, A, Y, s_0)$ is a sequence of edges $e_1, \cdots, e_n$ of $\Gamma$ such that the terminal vertex of $e_i$ is the initial vertex of $e_{i+1}$ for each $i \in \{1, \cdots, n-1\}$. A word $w = a_1 \cdots a_n$ in $A$ is \emph{read} by a path $e_1, \cdots, e_n$ in $\mathcal{G}$ if the label of $e_i$ is $a_i$ for each $i \in \{1, \cdots, n \}$. The \emph{language accepted by} $\mathcal{G}$ is the set of words in $A$ that are read by paths in $\mathcal{G}$ which start at the initial state $s_0$ and end at an accept state of $\mathcal{G}$. A language is \emph{regular} if it is the accepted language of some finite state automaton. 
\end{definition}

There is a growth function associated to a language.
\begin{definition}
    Given a language $L$ with alphabet $A$, the \emph{growth function} of $L$ is the function $f_{L}(n) : \mathbb{N} \rightarrow \mathbb{N}$ where
    \begin{align*}
        f_L(n) = |\{w \in L | \ell(w) \leq n\}|\text{.}
    \end{align*}
    The \emph{growth rate} of $L$ is
    \begin{align*}
        \lambda_{L} := \limsup_{n \rightarrow \infty} \sqrt[n]{f_{L}(n)}\text{.}
    \end{align*}
\end{definition}

\begin{theorem}\cite[Theorem~7]{Cannon-regular-geodesics}\label{thm:polynomial-growth}
    If $L$ is a regular language, there exist polynomials $P(x), Q(x) \in \mathbb{Q}[x]$ such that
    \begin{align*}
        \sum_{n=0}^{\infty} f_L(n) \cdot x^n  = \frac{P(x)}{Q(x)}\text{.}
    \end{align*}
\end{theorem}

\begin{definition}
    Given a finite directed graph $\Gamma$ with vertices $v_1, \cdots, v_n$, the \emph{adjacency matrix} for $\Gamma$ is the $(n \times n)$ matrix whose $(i,j)$th entry is 1 if there exists a directed edge connecting vertex $v_i$ to vertex $v_j$ and 0 otherwise. The \emph{adjacency matrix for a finite state automaton} $\mathcal{G} = (\Gamma, A, Y, s_0)$ is the adjacency matrix for the directed graph $\Gamma$. For either a directed graph $\Gamma$ or a finite state automaton $\mathcal{G}$, the \emph{Perron--Frobenius eigenvalue} $\rho_{\Gamma}$ or $\rho_{\mathcal{G}}$ is the eigenvalue of the adjacency matrix with largest absolute value.
\end{definition}

\begin{definition}
    A finite state automaton is \emph{pruned} if every state is the vertex of some path from the initial state to an accept state.
\end{definition}
Any finite state automaton can be pruned without altering the accepted language by removing any state which does not appear along any path from the initial state to an accept state. The Perron--Frobenius eigenvalue can be used to compute the growth rate of a regular language.

\begin{theorem}\cite[Theorem~3.6]{DFW}\label{thm:DFW-3-6}
    Let $\mathcal{G}$ be a pruned finite state automaton that accepts the regular language $L$. We have $\rho_{\mathcal{G}} \geq 1$ and
    \begin{align*}
        \rho_{\mathcal{G}} = \lim_{n \rightarrow \infty} \sqrt[n]{f_L(n)} = \lambda_L\text{.}
    \end{align*}
\end{theorem}
When the directed graph associated to one finite state automaton is contained in another, one can relate their Perron--Frobenius eigenvalues.
\begin{lemma}\cite[Lemma~3.2 and Theorem~3.4]{DFW}
    Let $\Gamma$ be a directed graph, and let $\Gamma'$ be a proper subgraph of $\Gamma$. If for every pair of distinct vertices $v,w \in \Gamma$, there exists a directed path in $\Gamma$ from $v$ to $w$ and $w$ to $v$, then $\rho_{\Gamma'} < \rho_{\Gamma}$.
\end{lemma}

We are interested in languages and automata because of the following connection with growth rates of subgroups. Corollary \ref{cor:growth-by-automaton} follows from Theorem \ref{thm:DFW-3-6}.
\begin{definition}
    Let $G$ be a group with finite generating set $A$, and let $L$ be a regular language with alphabet $A$. The language $L$ is \emph{geodesic} if for all $w \in L$, $w$ is a geodesic word in $\text{Cay}(G,A)$. It \emph{bijects with a subgroup} $H \leq G$ if for each $w \in L$, $\overline{w} \in H$ and the map $L \rightarrow H$ given by $w \rightarrow \overline{w}$ is a bijection.
\end{definition}
\begin{corollary}\label{cor:growth-by-automaton}
    Let $G$ be a finitely generated group with finite generating set $A$. Suppose there is a regular geodesic language $L$ with alphabet $A$ that bijects with a subgroup $H \leq G$. If $\mathcal{G}$ is a pruned finite state automaton that accepts $L$, then
    \begin{align*}
        \lambda_{H,A} = \lambda_L = \rho_{\mathcal{G}}\text{.}
    \end{align*}
\end{corollary}

Lastly, we will need to prove a straightforward lemma. It was shown in \cite{CRSZ} that if $H'$ is a finite-index subgroup of $H$, then $\lambda_{H,A} = \lambda_{H',A}$. In our setting, it will sometimes be necessary to pass to quotients by finite normal subgroups rather than to finite-index subgroups. We first make the following observation.

\begin{lemma}
    Let $H$ be a stable subgroup of a group $G$, and let $N$ be a finite normal subgroup of $G$. The image of $H$ under the quotient map $q : G \rightarrow G/N$ is a stable subgroup of $G/N$.
\end{lemma} 
\begin{proof}
    The quotient map $q: G \rightarrow G/N$ is both a homomorphism and a quasi-isometry. Stability is a quasi-isometry invariant, so the former condition implies that $q(H)$ is a subgroup, and the latter implies that it is stable.
\end{proof}
We can now prove the desired lemma.

\begin{lemma}\label{lem:quotient-by-finite-normal}
    Let $G$ be a group with finite generating set $A$, and let $H \leq G$. Let $N$ be a finite normal subgroup of $G$, and let $q : G \rightarrow G/N$ be the quotient map. The growth rates $\lambda_{H,A}$ and $\lambda_{q(H), q(A)}$ are equal.
\end{lemma}
We highlight that $H$ need not be a proper subgroup.
\begin{proof}
    Let $h \in H$. First, notice that in general, $d_{A}(1,h) \geq d_{q(A)}(1, h N)$. Since for $h$ in the alphabet $A$ also yields a representative for the coset $hN$ in the alphabet $q(A)$, although if an element of $A$ is contained in $N$, the length may shorten. Thus every element of $B_{n,A}(e) \cap H$ yields an element of $B_{n, q(A)}(e) \cap q(H)$. Similarly, if $d_{q(A)}(1, hN) = n$ for some $n$, then the word given by lifting each element of $q(A)$ in a word for $hN$ to the corresponding element of $A$ has the same length and represents an element of $hN$. If $K = \max_{n \in N} d_A (e,n)$, this implies $d_A (1, h) \leq n +K$. Thus every element of $B_{n , q(A)}(e) \cap q(H)$ yields at least one element of $B_{(n+K), A}(e) \cap H$.
    
    Finally, notice that in the quotient, each $h \in G$ is identified with $|N|$ other elements, any of which may be contained in $B_{n,A}(e) \cap H$. Thus in one direction, we obtain a bound that $|B_{n,A}(e) \cap H| \leq |N||B_{n, q(A)}(e) \cap q(H)|$ and in the other, we obtain $|B_{n+K , A} (e) \cap H | \geq |B_{n , q(A)} (e) \cap q(H)|$. Since $K$ and $|N|$ do not depend on $n$, this implies $\lambda_{H,A} = \lambda_{q(H), q(A)}$. 
\end{proof}

\section{Regular languages}

To prove Theorem \ref{thm:regular-language} and Corollary \ref{corx:rational-growth}, we will apply the following theorem of \cite{DFW} to construct a regular language of geodesic words representing elements of our stable subgroup $H$.

\begin{theorem}{\cite[Theorem~6.7]{DFW}}\label{thm:DFW-regular-lang}
        Let $G$ be a hyperbolic group acting properly and cocompactly on a graph $\Upsilon$. Fix a basepoint $b \in \Upsilon$. There is a regular language $J_G$ with the following properties:
        \begin{enumerate}
            \item $J_G \rightarrow G \cdot b$ is a surjection with fibers of cardinality exactly $\mid \text{Stab}_G(b)\mid$.
            \item The words in $J_b$, i.e., the preimage of $b$, have length 1 or 0, and correspond to paths of length 0.
            \item Every word of length $n$ in $J_G - J_b$ corresponds to a length $n$ geodesic in $\Upsilon$, starting at $b$.
            \item For every quasi-convex subgroup $H\subset G$, the sublanguage $J_H \subset J_G$ of words mapping to $Hb$ is regular.
            \item Let $\rho_H$ be the Perron--Frobenius eigenvalue of the transition matrix for any pruned automaton accepting $J_H$. The growth rate $\lambda_H(\Upsilon)$ satisfies 
            \begin{align*}
                \Lambda_H(\Upsilon) = \lim_{n \rightarrow \infty} \sqrt[n]{f_{H,\Upsilon}(n)} =  \lim_{n \rightarrow \infty} \sqrt[n]{f_{L_H}(n)} = \rho_H \text{.}
            \end{align*}
        \end{enumerate}
    \end{theorem}

Stable subgroups are always hyperbolic. Of course, $H$ acts geometrically on $\text{Cay}(H,S)$ for any finite generating set $S \subseteq H$, but the choice of graph impacts the alphabet of the regular language obtained. Since we require a regular language in the alphabet $A$, which need not be contained in $H$, we will need $H$ to act geometrically on a subgraph of $\text{Cay}(G,A)$.

\subsection{A geometric action of $H$}\label{subsec:geom-action-of-H}

Throughout the section, let $G$ be a group with finite generating set $A$, and let $H$ be an $(M,k)$-stable subgroup of $G$. Define $\Upsilon$ to be the union of all geodesics in $\text{Cay}(G,A)$ which start and end in $H$. 
\begin{lemma}
    $H$ acts freely and cocompactly on $\Upsilon \subseteq \Gamma(G,A)$.
\end{lemma}
\begin{proof}
    Let $x \in \text{Cay}(G,A)$ lie on a geodesic $\gamma$ from $p \in H$ to $q \in H$. The entire group $G$ acts by isometries, so for any $h \in H$, $h \cdot \gamma$ is a geodesic from $h \cdot p \in H$ to $h \cdot q \in H$ containing $h \cdot x$. Thus left-multiplication gives a group action $H \curvearrowright \Upsilon$. That this action is free follows immediately from the fact that $G \curvearrowright \text{Cay}(G,A)$ is free.
    
    Let $K = \Upsilon \cap B_{k+1, A}(e)$. Since $B_{k+1, A}(e)$ contains finitely many vertices and edges, so does $K$. Let $x_1$ and $x_2$ be two vertices connected by an edge $s$ all of which lie on a geodesic $\gamma \subset \Upsilon$ as before. Since $H$ is stable, there is some $h \in H$ such that $d_A (h, x_1) \leq k$. Since $H$ acts by isometries, $d_A (e, h^{-1} \cdot x) \leq k$. By construction, $d_A (e , h^{-1} x_2) \leq k+1$. The edge between the $h^{-1} \cdot x_i$, namely $h^{-1} \cdot s$, lies on the geodesic $h^{-1} \cdot \gamma$, so it is in $K$. Thus $\Upsilon = H \cdot K$ for a compact set $K$.
\end{proof}

\subsection{A regular language from the action}\label{subsec:regular-language-alphabet}
    We must now verify that the regular language obtained from this action has the correct alphabet. We also make explicit a detail which does not appear in \cite[Theorem~6.7]{DFW} but does follow from their method of proof.
    \begin{lemma}\label{lem:right-alphabet}
        In the case of $H$ acting on $\Upsilon$ as defined in the previous section, Theorem \ref{thm:DFW-regular-lang} gives a regular language with alphabet $A$ which bijects $H$. Furthermore, the language of geodesic words in $A$ which represent elements of $H$ is regular.
    \end{lemma}

    Before we begin, we recall the following theorem of Gersten and Short.
    \begin{theorem}\cite[Theorem~2.2]{Gersten-Short}
        Let $G$ be a hyperbolic group with an automatic structure $L_G$, and let $H$ be a quasi-convex subgroup. The sublanguage $L_H \subset G$ of words which map $H$ to itself is a regular language.
    \end{theorem}
    We can now prove the lemma.
    \begin{proof}
        Essentially, we proceed by carefully examining the proof in \cite{DFW} of Theorem \ref{thm:DFW-regular-lang}. First, we note that Construction 6.1 of \cite{DFW} is not necessary. Construction 6.1 replaces $\Upsilon$ with a graph $\Upsilon^{*}$ on which the group acts freely, but in our setting, $H$ already acts freely on $\Upsilon$. Thus we can begin with Construction 6.4, which replaces the group $H$ with a group $H^{+}$ which acts transitively on a graph $\Upsilon^{+}$. The subgroup $H$ does \emph{not} act transitively on $\Upsilon$ in our setting; some vertices in $\Upsilon$ are elements of $H$ while some are not, and these cannot lie in the same $H$-orbit.

        \textbf{Construction 6.4:} Attach 2-cells to $\Upsilon$ as follows. Choose a single representative from each $H$-orbit of based cycles, and attach a 2-cell along it. Extend this $H$-equivariantly to obtain a simply connected 2-complex with a free $H$-action. Call this $\Upsilon^{*}$. Let $D$ be the quotient of $H \setminus \Upsilon^{*}$ obtained by identifying all 0-cells. Define $H^{+} := \pi_1(D) \cong H * F_r$, where $F_r$ is a free group whose generators are in one-to-one correspondence with edges in a spanning tree for $H \setminus \Upsilon^{*}$. These come from orbits of edges in $\Upsilon$ which are not contained in $H$. Consider $\Tilde{D}$, the universal cover of $D$, which is a tree of copies of $\Upsilon^{*}$. Let $\Upsilon^{+}$ be the 1-skeleton of $\Tilde{D}$. The group $H^{+}$ is a deck group, which acts transitively on the vertices of $\Upsilon^{+}$.

        We can now proceed to Construction 6.5, which builds the regular language. 

        \textbf{Construction 6.5:} Choose a generating set $S^{+}$ for $H^{+}$ as follows: each generator corresponds to a closed path in the $1$-skeleton of $D$ consisting of a single edge. It is clear from the construction that $S^{+}$ generates $H^{+}$; $D$ has only one vertex, so every path in the fundamental group is a concatenation of the loops in $S^{+}$.
        
        This differs very slightly from the proof in \cite{DFW}; there, the generators are of the form $y e y'$ where $e$ is an edge in our sense and $y$ and $y'$ are a separate class of edges called \emph{tiny edges}. The tiny edges are added in Construction 6.1 to create a free action, and we did not need to implement that construction, so our complex $D$ does not have any tiny edges. The set $S^{+}$ is symmetric by construction.

        Recall that $D$ is a quotient of $\Upsilon^{*}$ where all 0-cells are identified, and $\Upsilon^{*}$ has $\Upsilon$ as its 1-skeleton. In particular, each edge in $D$ was an edge in $\Upsilon$, so each element of $S^{+}$ was a single element of $A$. The set $S^{+}$ is finite because the action of $H$ on $\Upsilon$ was cocompact.

        Since $H^{+} \cong H * F_r$ is hyperbolic, \cite[Theorem~11.27]{Cannon-regular-language} provides a geodesic regular language $L^{+}$ which maps bijectively to $H^{+}$. The alphabet of this language is the generating set $S^{+}$. Let $J_H$ be the sublanguage mapping bijectively to $H \subseteq H^{+}$. Since $H$ is quasi-convex in $H^{+} \cong H * F_r$, the sublanguage $J_H$ is regular by \cite[Theorem2.2]{Gersten-Short}.

        Finally, we note that the language of all geodesics with respect to a given finite generating set in a hyperbolic group is regular \cite{Cannon-regular-geodesics}. Again, $H$ is quasi-convex in the hyperbolic group $H^{+}$, so the sublanguage of geodesic words in the alphabet $S^{+} \subseteq A$ which represent elements of $H$ is also regular by \cite[Theorem~2.2]{Gersten-Short}. By construction, $S^{+}$ includes every letter of $A$ which appears on any geodesic between elements of $H$, so this is precisely the language of geodesic words in $A$ representing elements of $H$.
 \end{proof}

We can now prove Theorem \ref{thm:regular-language}.\smallskip

\begin{duplicate}[Theorem~\ref{thm:regular-language}]
    Let $G$ be a group generated by the finite set $A$, and let $H$ be an $(M,k)$-stable subgroup. For a subgroup $H \leq G$, let $L_H$ be the language of geodesic words in $A^{\star}$ representing elements of $H$. The subgroup $H$ is stable if and only if $L_H$ is a regular language all of whose elements are $M$-Morse. Moreover, if $H$ is stable, there is a sublanguage $J_H \subseteq L_H$ such that $J_H$ is regular, words of $J_H$ are $\text{Cay}(G,A)$ geodesics representing elements of $H$, and each element of $H$ is represented by exactly one word in $J_H$.
\end{duplicate}
    \begin{proof}
        When $G$ is Morse local-to-global, this is precisely \cite[Theorem~E]{CRSZ}. A careful inspection of the proof reveals that one direction, namely that $H$ is stable if $L_H$ is a regular language all of whose elements are $M$-Morse, does not use the Morse local-to-global assumption.
        
        Regularity and existence of geodesic languages $L_H$ and $J_H$ for an $(M,k)$ stable subgroup are given by \cite[Theorem~6.7]{DFW} and Lemma \ref{lem:right-alphabet}. Geodesics with endpoints in $H$ are $M$-Morse by the definition of an $(M,k)$-stable subgroup.
    \end{proof}

As in \cite{CRSZ}, we obtain the following corollary.
\begin{corollary}
    Let $H$ be an infinite-index stable subgroup of a group $G$ with finite generating set $A$, and let $f_{H,A}(n) = |B_n(e) \cap H|$. There exist polynomials $P(x)$ and $Q(x) \in \mathbb{Q}[x]$ such that 
    \begin{align*}
        \sum_{n=0}^{\infty} f_{H,A} \cdot x^n = \frac{P(x)}{Q(x)}\text{.}
    \end{align*}
\end{corollary}
\begin{proof}
    Let $J_H$ be the regular language which bijects with $H$ of Theorem \ref{thm:regular-language}. By construction, the growth functon of $J_H$ is precisely $f_{H,A}(n)$. The result follows from Theorem \ref{thm:polynomial-growth}.
\end{proof}

\section{Growth rates}

We can now prove Theorem \ref{thmx:growth-rate-gap} of the introduction. Our proof is essentially the same as the proof of Theorem A in \cite{CRSZ} but with our Theorem \ref{thm:regular-language} and an additional condition on $H$ used in place of the analogous results for Morse local-to-global groups. We reproduce the argument to ensure it is clear to the reader that these substitutions are sufficient. 

\begin{theorem}\label{thm:growth-rate-gap}
    Let $G$ be a group with finite generating set $A$, and let $H$ be an infinite index stable subgroup such that there is an infinite order element $g \in G$ with $\langle H , g \rangle \cong H * \langle g \rangle$. Then 
    \begin{align*}
        \lambda_{H,A} < \lambda_{G,A}\text{.}
    \end{align*}
\end{theorem}
\begin{proof}
    By Theorem \ref{thm:regular-language}, there is a regular language $J_H$ with alphabet $A$ which bijects $H$. By Corollary \ref{cor:growth-by-automaton}, $\lambda_{H,A} = \rho_{\mathcal{G}}$ where $\mathcal{G}$ is a pruned finite-state automaton that accepts $J_H$.

    Let $w$ be a geodesic word in $A^{\star}$ that represents the element $g$ such that $\langle H , g \rangle \cong H * \langle g \rangle$, and let $s_0$ be the unique initial state of $\mathcal{G}$. From $\mathcal{G}$, we construct a new finite-state automaton $\mathcal{G'}$ as follows:
    \begin{itemize}
        \item The states of $\mathcal{G}'$ are precisely the states of $\mathcal{G}$. The initial state is $s_0$.
        \item The accept states of $\mathcal{G}'$ are the accept states of $\mathcal{G}$ plus $s_0$.
        \item If there is a directed edge of $\mathcal{G}$ from the state $s$ to the state $t$, then there is also a directed edge of $\mathcal{G}'$ from $s$ to $t$.
        \item For each accept state $s$ of $\mathcal{G}$, $\mathcal{G}'$ has an additional directed edge with label $w$ starting at $s$ and ending at $s_0$. This does include a directed edge starting and ending at $s_0$.
    \end{itemize}

    Let $L'$ be the language accepted by the new automaton $\mathcal{G}'$. Since $\mathcal{G}'$ is obtained from $\mathcal{G}$ by adding paths between accept states of $\mathcal{G}'$, $\mathcal{G}$ being pruned implies $\mathcal{G}'$ is pruned. Thus $\rho_{\mathcal{G}} = \lambda_{J_H}$ and $\rho_{\mathcal{G}'} = \lambda_{L'}$. In particular, \cite[Lemma~2.14]{CRSZ} or \cite[Lemma~3.2 and Theorem~3.4]{DFW} imply that 
    \begin{align*}
        \lambda_{J_H} = \rho_{\mathcal{G}} < \rho_{\mathcal{G}'} = \lambda_{L'}\text{.}
    \end{align*}
    Since $\langle H, g \rangle \cong H * \langle g \rangle$, the map $L' \rightarrow G$ given by $u \rightarrow \overline{u}$ is injective and hence $L'$ is a regular language with alphabet $A$ that bijects $\langle H , g \rangle$. Thus by \cite{CRSZ} Corollary 2.17, 
    \begin{align*}
        \lambda_{H,A} = \lambda_{J_H} = \rho_{\mathcal{G}} < \rho_{\mathcal{G}'} = \lambda_{L'} = \lambda_{\langle H , g \rangle, A}\text{.}
    \end{align*}
    By definition, $\lambda_{\langle H , g \rangle, A} \leq \lambda_{G,A}$, so $\lambda_{H,A} < \lambda_{G,A}$.
\end{proof}

\section{Transverse loxodromics}
Elements $g$ suitable for applying Theorem \ref{thmx:growth-rate-gap} are known to exist in many cases; see, for example \cite{Abbott-Dahmani,reducibles-in-CAL,Arzhantseva-qconvex-hyperbolic,cstar-simple-rel-hyp,Gitik-ping-pong,Gromov-hyp-groups,Legaspi-path-system-growth,RST}. We restrict our attention to settings where a growth rate gap was not already known. Specifically, an element $g$ suitable for applying Theorem \ref{thmx:growth-rate-gap} exists when $G$ is acylindrically hyperbolic and $H$ is sufficiently well-behaved with respect to some purely WPD action by work of Abbott and Hull. First, we recall the definition of a transverse loxodromic element.
\begin{definition}
    Let a group $G$ act on a hyperbolic metric space $X$, and let $\Sigma$ be a subset of $G$. A loxodromic element $f$ is \emph{transverse} to $\Sigma$ if $f$ has a quasi-geodesic axis $\alpha_f$ in $X$ such that for all $K > 0$, there exists $L \geq 0$ such that $\text{diam}(\alpha_f \cap N_K(g \pi(\Sigma))) \leq L$ for all $g \in G$, where $\pi$ is the orbit map with respect to some basepoint.
\end{definition}
We can now state the theorem.
\begin{theorem}\cite[Special case of Theorem~1.1]{Abbott-Hull}\label{thm:abbott-hull-1.1}
    Let $G$ be a group with a non-elementary, partially WPD action on a hyperbolic metric space $X$, and let $H$ be a subgroup of $G$ such that $H \cap E(G) = \{1\}$ where $E(G)$ is the unique maximal normal subgroup of $G$. Suppose $H$ has quasi-convex orbits in $X$ and there exists a loxodromic WPD element $f$ transverse to $H$. There is an infinite subgroup $R$ such that $\langle H , R \rangle \cong H * R$ and $\langle H, R \rangle$ has quasi-convex orbits in $X$. If $H$ is quasi-isometrically embedded in $X$, then $\langle H , R \rangle $ is quasi-isometrically embedded in $X$.
\end{theorem}
In fact, the original theorem gives the much stronger statement that random subgroups $R$ satisfy these properties, but we require only existence. We thus obtain the following as a corollary of Theorem \ref{thm:growth-rate-gap} and Theorem \ref{thm:abbott-hull-1.1}.

\begin{corollary}\label{cor:trans-lox-growth-gap}
    Let $G$ be a finitely generated group acylindrically hyperbolic group, and $H$ be an infinite-index stable subgroup. If there is a non-elementary and partially WPD action of $G$ on a hyperbolic metric space $X$ such that $H$ has quasi-convex orbits and there exists a loxodromic WPD element $g$ which is transverse to $H$, then for any finite generating set $A$ of $G$,
    \begin{align*}
        \lambda_{H,A} < \lambda_{G, A}\text{.}
    \end{align*}
\end{corollary}
\begin{proof}
     Recall that every acylindrically hyperbolic group has a unique maximal normal subgroup $E(G)$, which is necessarily finite. The quotient $G/E(G)$ is also acylindrically hyperbolic; indeed, any non-elementary partially WPD action $G \curvearrowright X$ induces a non-elementary and partially WPD action $G/E(G) \curvearrowright X$ in a straightforward way. Furthermore, passing to the quotient preserves the property of a subgroup having quasi-convex orbits. A suitable free product for the image of $H$ in $G/E(G)$ is then guaranteed by Theorem \ref{thm:abbott-hull-1.1}. Applying Lemma \ref{lem:quotient-by-finite-normal} to both $H$ and $G$ completes the proof. 
\end{proof}
In Section 6.4 of \cite{Abbott-Hull}, the authors construct a transverse WPD element for any convex cocompact subgroup of $\Out(F_n)$ in the action on the free factor graph when $n \geq 3$. The following is thus a special case of Corollary \ref{cor:trans-lox-growth-gap}.
\begin{corollary}
    Let $A$ be any finite generating set of $\Out(F_n)$ for $n \geq 3$. For any convex cocompact subgroup $H$, $\lambda_{H,A} < \lambda_{\Out(F_n), A}$.
\end{corollary}

Transverse loxodromics are constructed for many more examples in \cite{Abbott-Hull}, but in all remaining cases, either the ambient group $G$ is known to be Morse local-to-global or the stable subgroup can be made to act elliptically in an acylindrical action. Growth rate gaps were already known in either setting \cite{CRSZ,DFW}. We can, however, expand on their methods to obtain new examples. In particular, \cite[Proposition~6.9]{Abbott-Hull} shows that transverse loxodromics exist for any stable subgroup of $\mathrm{MCG}(S)$ in the natural action on the associated curve graph, $C(S)$, using the connection between $C(S)$ and Teichm\"{u}ller space. Since Theorem \ref{thm:abbott-hull-1.1} and, consequently, Corollary \ref{cor:trans-lox-growth-gap} do not require the partially WPD action to be cobounded, similar methods can be used to prove the following.


\begin{proposition}\label{prop:subgroup-transverse}
    Let $G$ be a finitely generated subgroup of $\mathrm{MCG}(S)$ for a non-exceptional surface $S$ such that the induced action $G \curvearrowright C(S)$ is non-elementary, and let $H$ be a convex cocompact subgroup of $\mathrm{MCG}(S)$ which is also an infinite-index subgroup of $G$. There is a loxodromic element $x$ in the action $G \curvearrowright C(S)$ which is transverse to $H$.
\end{proposition}
By non-exceptional, we mean of sufficient complexity that $C(S)$ is connected. The convex cocompact subgroups of the mapping class group are exactly the stable subgroups, so \cite[Theorem~1.6]{Aougab-Durham-Taylor-OutFn} shows that $H$ is also stable in $G$ whenever $G$ is finitely generated. It is necessarily infinite-index whenever $G$ is not hyperbolic. The Torelli group and the handlebody group are both finitely generated except in exceptional cases, and neither is hyperbolic \cite{Wajnryb-handlebody-pres,Suzuki-handlebody-pres,Johnson-Torelli-one,Johnson-Torelli-three}. Thus we obtain the following corollaries.
\begin{corollary}
    Let $\mathcal{T}_g$ be the Torelli group with genus $g >2$, and let $H$ be a convex cocompact subgroup of $\mathrm{MCG}(S_g)$ with $H < \mathcal{T}$. For any finite generating set $A$ of $\mathcal{T}$, $\lambda_{H,A} < \lambda_{\mathcal{T}, A}$.
\end{corollary}
\begin{corollary}
    Let $\mathcal{H}_g$ be a handlebody group of genus $g \geq 2$. Let $H$ be a convex cocompact subgroup of the mapping class group of the associated boundary surface with $H < \mathcal{H}_g$. For any finite generating set $A$ of $\mathcal{H}_g$, $\lambda_{H,A} < \lambda_{\mathcal{H}_g, A}$.
\end{corollary}
In the special case that $g = 2$, the latter corollary was already known as a consequence of the facts that $\mathcal{H}_2$ is an HHG and HHGs are Morse local-to-global \cite{Chesser-handlebody-not-HHG,RST}. This is truly a special case, however; when $g > 2$, Hamenst\"{a}dt and Hensel showed that $\mathcal{H}_g$ has exponential Dehn function, and hierarchically hyperbolic groups have at quadratic Dehn functions \cite{hamenstadt-hensel-handlebody,HHS2}.

Before we proceed to the proof, we will need the following straightforward observation.
\begin{lemma}\label{lem:transverse-in-subgroup}
    Let $G$ be a finitely generated group, let $\Sigma$ be a subset of $G$, and let $X$ be a hyperbolic metric space such that $G \curvearrowright X$. Let $H$ be any subgroup of $G$ such that $\Sigma \cap H \neq \{1\}$. If $f \in H$ is loxodromic and transverse to $\Sigma$ with respect to $G \curvearrowright X$, then $f$ is also transverse to $\Sigma \cap H$ with respect to the induced action $H \curvearrowright X$.
\end{lemma}
\begin{proof}
    By the definition of transversality, $f$ has a quasi-axis $\alpha_f$ in $X$ such that for all $K >0$, there exists an $L \geq 0$ such that $\text{diam}(\alpha_f \cap \mathcal{N}_K(g \pi(\Sigma)) \leq L$ for all $g \in G$. The set $\pi(\Sigma \cap H)$ is contained in $\pi(\Sigma)$, and every element of $H$ is also an element of $g$, so the same $L$ and $K$ prove the desired statement. 
\end{proof}

The proof of Proposition \ref{prop:subgroup-transverse} is very similar to the proof of \cite[Proposition~6.9]{Abbott-Hull}, but there are some subtleties, and we address these explicitly for clarity. Recall that \cite[Theorem~5.1]{Abbott-Hull} constructs, for any finitely generated group $G$ with a non-elementary WPD action on a hyperbolic space $X$ and any infinite-index stable subgroup $H < G$, a loxodromic element $f \in G$ such that $g H g^{-1} \cap \langle f \rangle = \{1\}$ for all $g \in G$. We first make the following observation.
\begin{lemma}\label{lem:pA-in-G}
    For $H$ and $G$ as in Proposition \ref{prop:subgroup-transverse}, the loxodromic element $f$ of \cite[Theorem~5.1]{Abbott-Hull} can be chosen to be a pseudo-Anosov in $G$.
\end{lemma}
\begin{proof}
    Given an acylindrical action $\mathrm{MCG}(S) \curvearrowright C(S)$, any subgroup $G < \mathrm{MCG}(S)$ also acts acylindrically on $C(S)$. In general, this action may be elementary, but we have assumed this is not the case, so $G \curvearrowright C(S)$ is non-elementary and partially WPD. Then \cite[Theorem~5.1]{Abbott-Hull} automatically gives an element in $G$. Since $f$ acts loxodromically on $C(S)$, $f$ is pseudo-Anosov.
\end{proof}

Next, we address a subtlety. The statement of \cite[Lemma~6.7]{Abbott-Hull} claims that for any convex cocompact subgroup $H < \text{MCG}(S)$, there exists a pseudo-Anosov element $f$ and a constant $B= B(f,H)$ such that $\text{diam}_{\mathcal{T}}(p_f (gH\cdot x_0)) \leq B$ for all $g \in \text{MCG}(S)$, where $\mathcal{T}$ denotes the Teichm\"{u}ller space associated to $S$ with the Teichm\"{u}ller metric, $x_0$ is some basepoint in $\mathcal{T}$, and $p_f$ denotes closest point projection to an axis $\alpha_f$ of the pseudo-Anosov element $f$. The proof is constructive, and it uses the element produced in \cite[Theorem~5.1]{Abbott-Hull}. While we can guarantee this element $f$ is in $G$, we cannot immediately guarantee that $g H g^{-1} \cap \langle f \rangle = \{1\}$ for all $g \in \text{MCG}(S)$ rather than all $g \in G < \text{MCG}(S)$. This does not cause any structural issues in adapting \cite[Lemma~6.7]{Abbott-Hull} to our setting, but we reproduce the proof to avoid burdening the reader with independently verifying this.
\begin{lemma}\label{lem:5.6-analogue}
    Let $G$ and $H$ be as in Proposition \ref{prop:subgroup-transverse}. For any pseudo-Anosov element $f \in G$ such that $g H g^{-1} \cap \langle f \rangle = \{1\}$ for all $g \in G$, there is some constant $B = B(f,H)$ such that $\text{diam}_{\mathcal{T}}(p_f (gH\cdot x_0)) \leq B$ for all $g \in G$.
\end{lemma}
\begin{proof}
    Let $\mathcal{T}(S)$ denote the Teichm\"{u}ller space with Teichm\"{u}ller metric, and fix a basepoint $x_0 \in \mathcal{T}(S)$. The action $\text{MCG}(S) \curvearrowright \mathcal{T}(S)$ is proper by \cite{MCG-primer}, so the action $G \curvearrowright \mathcal{T}(S)$ is proper as well. Let $\nu$ be a constant so that $gH \cdot x_0$ is $\nu$-quasi-convex in $\mathcal{T}(S)$ for all $g \in G$. Let $f$ be as in the statement. 
    
    Since $f$ is pseudo-Anosov, the axis $\alpha_f$ is strongly contracting by \cite{Minsky-Teich-projections}, meaning there exist constants $B'$ and $K$ such that for all geodesics $\gamma$ at distance at least $K$ from $\alpha_f$, $\text{diam}_{\mathcal{T}}(p_f(\gamma)) \leq B'$. Thus it suffices to show that there exists $B = B(f,h)$ such that for all $g \in G$ and all $x_1 , x_2 \in \alpha_f$ with $d_{\mathcal{T}}(x_i , gH\cdot x_0) \leq K$, we have $d_{\mathcal{T}}(x_1 , x_2) \leq B$. 
    
    Suppose not. Then for any $D$, there exists $g \in G$ and orbit points $f^{i_1} \cdot x_0 , f^{i_2} \cdot x_0 \in \alpha_f$ with $d_{\mathcal{T}}(f^{i_j} \cdot x_0 , gH \cdot x_0) \leq K$ for $j = 1$ and $2$, and $d_{\mathcal{T}}(f^{i_1} \cdot x_0 , f^{i_1} \cdot x_0) \geq D$. Let $x_1$ and $x_2$ be the nearest points in $gH \cdot x_0$ to $f^{i_1} \cdot x_0$ and $f^{i_2} \cdot x_0$ respectively. By choosing $D$ sufficiently large, we can ensure that the concatenation $\gamma = [f^{i_1} \cdot x_0, x_1][x_1 , x_2] [x_2, f^{i_2}\cdot x_0]$ is a uniform quasi-geodesic with constants depending only on the quasi-constants of $\alpha_f$. Since $\alpha_f$ is strongly contracting, it is also Morse \cite{Arzhantseva-Cashen-Gruber-Hume-contraction}, so there is a constant $K'$ depending only on $K$ and the quasi-constants for $\alpha_f$ such that the Hausdorff distance between $\gamma$ and $\alpha_f \mid_{[f^{i_1} \cdot x_0 , f^{i_2} \cdot x_0]}$ is at most $K'$.

    Moreover, since $gH\cdot x_0$ is $\nu$-quasi-convex, the geodesic $[x_1 , x_2]$ is contained in the $\nu$-neighborhood of $gH \cdot x_0$. Thus for every $i_1 \leq i \leq i_2$, we have $d_{\mathcal{T}}(f^{i} \cdot x_0 , gH \cdot x_0) \leq K' + \nu$. Let $gh_i \cdot x_0 \in g H \cdot x_0$ be the nearest point in $gH \cdot x_0$ to $f^{i}x_0$ for each $i$, so that $d_{\mathcal{T}}(x_0 , f^{-i}gh_i \cdot x_0) \leq K' + \nu$ for each $i$. The action of $G$ on $\mathcal{T}(S)$ is proper and $D$ can be arbitrarily large, so for some $i \neq j$, we must have $f^{-i}gh_i = f^{-j}gh_j$. In particular, $f^{j-i} = g h_j h_i^{-1} g^{-1}$, which is a contradiction because $g H g^{-1} \cap \langle f \rangle = \{1\}$ for all $g \in G$.
\end{proof}

\begin{proof}[Proof of Proposition \ref{prop:subgroup-transverse}]
    Apply \cite[Theorem~5.1]{Abbott-Hull}, Lemma \ref{lem:pA-in-G}, and Lemma \ref{lem:5.6-analogue} to obtain a pseudo-Anosov element $f \in G$ and a constant $B$ such that $\text{diam}_{\mathcal{T}}(p_f (gH\cdot x_0)) \leq B$ for all $g \in G$. Upon careful inspection, the proof of \cite[Proposition~6.9]{Abbott-Hull}, which is Proposition \ref{prop:subgroup-transverse} when $G = \text{MCG}(S)$, in fact shows that for any $K$, there is a constant $C \geq 0$ such that $\text{diam}(\beta_f  \cap \mathcal{N}_K \pi(gH)) \leq C$ for a fixed $C(S)$ axis $\beta_f$ of $f$ whenever $H$, $g$, and $f$ satisfy the conclusion of Lemma \ref{lem:5.6-analogue} for a uniform $B$. This holds for all $g \in G$ with the $B$ of Lemma \ref{lem:5.6-analogue} by construction.
\end{proof}
To prove the corollaries, it suffices to note that the actions of $\mathcal{H}_g$ and $\mathcal{T}_g$ on $C(S_g)$ are non-elementary. Indeed, it is well-known that neither group is virtually cyclic and that each contains at least one pseudo-Anosov element of the mapping class group; see, for example \cite{Hensel-handlebody-primer,pA-generic-Torelli}. 

As a brief aside, the proof that a transverse loxodromic element exists for convex cocompact subgroups of $\Out(F_n)$ is almost exactly the same as for convex cocompact subgroups of $\mathrm{MCG}(S)$, and a nearly identical argument to the one above can be used to show an analogous result for convex cocompact subgroups contained in any finitely generated subgroup $G$ of $\Out(F_n)$ whose induced action on the free factor graph is non-elementary.

\section{On the membership problem for stable subgroups}
In \cite[Question~3]{CRSZ}, the authors asked whether the membership problem is decidable for stable subgroups of Morse local-to-global groups and whether their Theorem E, generalized by our Theorem \ref{thm:regular-language}, might have implications on the complexity. In this section, we provide an example which demonstrates that the answer is no. First, we recall the following.

\begin{theorem}\cite[Theorem~3]{Gitik-membership}
    Let $G$ be a finitely generated group with decidable word problem, and let $H$ be a quasi-convex subgroup of $G$. The membership problem for $H$ is decidable.
\end{theorem}
Stable subgroups are necessarily quasi-convex, so the membership problem for stable subgroups is decidable without the Morse local-to-global assumption in the case where $G$ has decidable word problem. One can easily construct a Morse local-to-global group with undecidable word problem from an arbitrary group with undecidable word problem.

\begin{lemma}\label{lem:undecideable-word}
    Let $G$ be a finitely presented group with undecidable word problem. The finitely presented groups $G \times \mathbb{Z}$ and $(G \times \mathbb{Z}) * \mathbb{Z}$ are Morse local-to-global groups with undecidable word problem. The former is Morse limited, and the latter is relatively hyperbolic and contains Morse elements.
\end{lemma}
\begin{proof}
    Recall from \cite{RST} that any group with an infinite central subgroup is Morse limited, so $G \times \mathbb{Z}$ is Morse limited and therefore Morse local-to-global. A free product of Morse local-to-global groups is Morse local-to-global, and it is straightforward to verify that if we choose generators $s$ and $t$ such that $(G \times \mathbb{Z}) * \mathbb{Z} \cong (G \times \langle s \rangle ) * \langle t \rangle$, the element $t$ is Morse. Indeed, any word in the free product is of the form $w_{1, (G \times \langle s \rangle)} t^i w_{2, (G \times \langle s \rangle)} t^j \cdots$ where $w_{i, (G \times \langle s \rangle}$ is any word in $G \times \langle s \rangle$, and such a word is a $(k,c)$ quasi-geodesic with endpoints on $\langle t \rangle$ precisely when each $w_i$ is a uniformly short loop in $G \times \langle s \rangle$. 

    The word problem is decidable in a direct (free) product precisely when it is decidable in each factor. Thus neither $G \times \mathbb{Z}$ nor $(G \times \mathbb{Z}) * \mathbb{Z}$ has decidable word problem.
\end{proof}

In a similar way, one can show the following.
\begin{proposition}
    Let $G$ be a finitely presented group with unsolvable word problem. The group $(G \times \mathbb{Z}) * \mathbb{Z}$ is a relatively hyperbolic and Morse local-to-global group which contains a stable subgroup with unsolvable membership problem.
\end{proposition}
\begin{proof}
    Let $(G \times \mathbb{Z}) * \mathbb{Z} \cong (G \times \langle s \rangle) * \langle t \rangle$ as before. As in Lemma \ref{lem:undecideable-word}, $t$ is a Morse element, i.e., $\langle t \rangle$ is stable. Consider the word $w_{(G \times \langle s \rangle)} t^k$, where $w_{(G \times \langle s \rangle)} \in G \times \langle s \rangle$ and $k$ is any integer. Since there are no relations between words $w \in G \times \langle s \rangle$ and powers of $t$, determining whether such an element belongs to $\langle t \rangle$ is precisely equivalent to determining if $w_{(G \times \langle s \rangle)}$ is the identity element of $G \times \langle s \rangle$, which is algorithmically undecideable by choice of $G$.
\end{proof}

\bibliographystyle{acm}
\bibliography{refs.bib} 

@article{reducibles-in-CAL,
  author = {Antol\'{i}n, Yago and Cumplido, Mar\'{i}a},
  title = {Parabolic subgroups acting on the additional length graph},
  journal = {Algebraic \& Geometric Topology},
  year = 2021,
  volume = 21,
  pages = {1791–1816}}

@article{HHS2,
  title={Hierarchically hyperbolic spaces {II}: {C}ombination theorems and the distance formula},
  author={Behrstock, Jason A. and Hagen, Mark F. and Sisto, Alessandro},
  journal={Pacific Journal of Mathematics},
  year={2015},
}

@article{DFW,
  author = {Dahmani, Fran\c{c}ois and Futer, David and Wise, Daniel},
  title = {Growth of quasiconvex subgroups},
  journal = {Mathematical Proceedings of the Cambridge Philosophical Society},
  year = 2019,
  volume = 163,
  pages = {505-530}}

@article{CRSZ,
  author = {Cordes, Matthew and Russell, Jacob and Spriano, Davide and Zalloum, Abdul},
  title = {Regularity of {M}orse geodesics and growth of stable subgroups},
  journal = {Journal of Topology},
  year = 2022,
  volume = 15,
  pages = {1217–1247}}

@article{RST,
  author = {Russell, Jacob and Spriano, Davide and Tran, Hung Cong},
  title = {The local-to-global property for {M}orse quasi-geodesics},
  journal = {Mathematische Zeitschrift},
  year = 2022,
  volume = 300,
  pages = {1557–1602}}

@incollection{Cannon-regular-language,
  booktitle={Ergodic theory, symbolic dynamics and hyperbolic spaces},
  title={The theory of negatively curved spaces and groups},
  author={Cannon, J.W.},
  publisher={Oxford Sci. Publ., New York},
  pages={315--369},
  year={1991}
}

@article{Gersten-Short,
  title={Rational subgroups of biautomatic groups},
  author={Gersten, S.M. and Short, H.B.},
  journal={Annals of Mathematics},
  volume={134},
  number={1},
  pages={125--158},
  year={1991}
}

@article{Abbott-Dahmani,
  title={Acylindrically hyperbolic groups have property ${P}_{naive}$},
  author={Abbott, Carolyn and Dahmani, Fran\c{c}ois},
  journal={Mathematiche Zeitschrift},
  volume={291},
  pages={555--568},
  year={2019}
}

@article{Abbott-Hull,
  title={Random walks and quasi-convexity in acylindrically hyperbolic groups},
  author={Abbott, Carolyn and Hull, Michael},
  journal={Journal of Topology},
  volume={14},
  number={3},
  pages={992--1026},
  year={2021}
}

@article{Growth-rate-gap-all-Mltg,
  title={Growth rate gap for stable subgroups},
  author={Han, Suzhen and Liu, Qing},
  journal={arXiv:2412.11244},
  year={2024}
}

@article{Durham-Taylor-stability,
  title={Convex cocompactness and stability in mapping class groups},
  author={Durham, Matthew Gentry and Taylor, Samuel J},
  journal={Algebraic \& Geometric Topology},
  volume={15},
  pages={2837--2857},
  year={2015}
}

@article{Cordes-Hume-stability,
  title={Stability and the {M}orse boundary},
  author={Cordes, Matthew and Hume, David},
  journal={Journal of the London Mathematical Society},
  volume={95},
  number={2},
  pages={963--988},
  year={2017}
}

@article{Aougab-Durham-Taylor-OutFn,
  title={Pulling back stability with applications to $\mathrm{{O}ut}({F}_n)$},
  author={Aougab, Tariq and Durham, Matthew and Taylor, Samuel},
  journal={Journal of the London Mathematical Society},
  volume={96},
  number={2},
  pages={565--583},
  year={2017}
}

@incollection{Gitik-membership,
  author={Gitik, Rita},
  title={Two algorithms in group theory},
  editor={Baginski, P. and Fine, B. and Moldenhauer, A. and Rosenberger, G. and Shpilrain, V.},
  booktitle={Elementary Theory of Groups and Group Rings, and Related Topics: Proceedings of the Conference held at Fairfield University and at the Graduate Center, CUNY},
  publisher={De Gruyter},
  address={Berlin, Boston},
  year={2020},
  pages={73-80},
  chapter={Two algorithms in group theory}
}

@article{Suzuki-handlebody-pres,
  title={On homeomorphisms of a 3-dimensional handlebody},
  author={Suzuki, Shin'ichi},
  journal={Canadian Journal of Mathematics},
  volume={29},
  number={1},
  pages={111--124},
  year={1977}
}

@article{Wajnryb-handlebody-pres,
  title={Mapping class group of a handlebody},
  author={Wajnryb, Bronis\law},
  journal={Fundamenta Mathematicae},
  volume={158},
  year={1998}
}

@article{Johnson-Torelli-one,
  title={The Structure of the {T}orelli Group {I}: A Finite Set of Generators for $\mathcal{I}$},
  author={Johnson, Dennis},
  journal={Annals of Mathematics},
  volume={118},
  number={3},
  pages={423--442},
  year={1983}
}

@article{Johnson-Torelli-three,
  title={The Structure of the {T}orelli Group {III}: The abelianization of $\mathcal{I}$},
  author={Johnson, Dennis},
  journal={Topology},
  volume={24},
  number={2},
  pages={127--144},
  year={1985}
}

@article{Chesser-handlebody-not-HHG,
  title={Stable subgroups of the genus two handlebody group},
  author={Chesser, Marissa},
  journal={Algebraic \& Geometric Topology},
  volume={22},
  number={2},
  pages={919--971},
  year={2022}
}

@article{Cannon-regular-geodesics,
  title={The combinatorial structure of cocompact discrete hyperbolic groups},
  author={Cannon, James W.},
  journal={Geometriae Dedicata},
  volume={16},
  number={2},
  year={1984}
}

@article{Gitik-ping-pong,
  title={Ping-pong on negatively curved groups},
  author={Gitik, Rita},
  journal={Journal of Algebra},
  volume={217},
  number={1},
  year={1999}
}

@article{Arzhantseva-qconvex-hyperbolic,
  title={On quasiconvex subgroups of word hyperbolic groups},
  author={Arzhantseva, G.N.},
  journal={Geometriae Dedicata},
  volume={87},
  pages={191--208},
  year={2001}
}

@inproceedings{Hensel-handlebody-primer,
  title={A PRIMER ON HANDLEBODY GROUPS},
  author={Hensel, Sebastian},
  year={2018},
  url={https://api.semanticscholar.org/CorpusID:211542546}
}

@article{pA-generic-Torelli,
  title={On genericity of pseudo-{A}nosovs in the {T}orelli group},
  author={Malestein, Justin and Souto, Juan},
  journal={International Mathematics Research Notices},
  volume={2013},
  issue={6},
  pages={1434--1449},
  year={2013}
}

@article{cstar-simple-rel-hyp,
  title={Relatively hyperbolic groups are ${C}^{*}$-simple},
  author={Arzhantseva, G. and Minasyan, A.},
  journal={Journal of Functional Analysis},
  volume={243},
  number={1},
  pages={345--351},
  year={2007}
}

@InProceedings{Gromov-hyp-groups,
    author = {Gromov, M.},
    title = {Hyperbolic groups},
    booktitle = {Essays in Group Theory},
    series = {Mathematical Sciences Research Institute Publications},
    year = 1987,
    publisher = {Springer New York, NY},
    volume={8},
    editor = {S.M. Gersten},
    pages = {75-263},
}

@article{hamenstadt-hensel-handlebody,
  title={The Geometry of the Handlebody Groups {I}{I}: {D}ehn Functions},
  author={Hamenst\"{a}dt, Ursula and Hensel, Sebastian},
  journal={Michigan Mathematical Journal},
  volume={70},
  number={1},
  pages={23--53},
  year={2021}
}

@book {MCG-primer,
    AUTHOR = {Farb, Benson and Margalit, Dan},
     TITLE = {A primer on mapping class groups},
    SERIES = {Princeton Mathematical Series},
    VOLUME = {49},
 PUBLISHER = {Princeton University Press, Princeton, NJ},
      YEAR = {2012},
     PAGES = {xiv+472},
      ISBN = {9780691147949}
}

@article{Minsky-Teich-projections,
  title={Quasi-projections in {T}eichm\"{u}ller space},
  author={Minsky, Yair},
  journal={Journal f\"{u}r die reine und angewandte {M}athematik},
  volume={473},
  pages={121--136},
  year={1996}
}

@article{Arzhantseva-Cashen-Gruber-Hume-contraction,
  title={Characterizations of {M}orse quasi-geodesics via superlinear divergence and sublinear contraction},
  author={Arzhantseva, Goulnara N. and Cashen, Christopher H. and Gruber, Dominik and Hume, David},
  journal={Documenta Mathematica},
  volume={22},
  pages={1193--1224},
  year={2017}
}

@article{Holt-Rees-qgeod-regular,
  title={Regularity of quasi-geodesics in a hyperbolic group},
  author={Holt, Derek and Rees, Sarah},
  journal={International Journal of Algebra and Computation},
  volume={13},
  number={05},
  pages={585--596},
  year={2003}
}

@article{Sam-Patrick-Davide-qgeod-regular,
  title={Regularity of quasi-geodesics characterizes hyperbolicity},
  author={Hughes, Sam and Nairne, Patrick and Spriano, Davide},
  journal={Proceedings of the Royal Society of Edinburgh Section A: Mathematics},
  pages={1--14},
  year={2025}
}

@article{Abdul-Alex-injective-contracting,
  title={Morse subsets of injective spaces are strongly contracting},
  author={Sisto, Alessandro and Zalloum, Abdul},
  journal={Groups, Geometry and Dynamics},
  volume={19},
  pages={1425--1443},
  year={2025}
}

@article{Legaspi-path-system-growth,
  title={Growth of quasi-convex subgroups in groups with
a constricting element},
  author={Legaspi, Xabier},
  journal={Groups, Geometry and Dynamics},
  volume={18},
  number={4},
  pages={1469--1505},
  year={2024}
}

\end{document}